\documentclass{amsart}
\usepackage[dvipsnames,svgnames,x11names,hyperref]{xcolor}
\usepackage{textgreek,bbm,url,graphicx,verbatim,amssymb,enumerate,stmaryrd,booktabs,lmodern,mathtools,microtype,mathabx}
\usepackage[pagebackref,colorlinks,citecolor=Mahogany,linkcolor=Mahogany,urlcolor=Mahogany,filecolor=Mahogany]{hyperref}
\usepackage[capitalize]{cleveref}
\usepackage[mathscr]{euscript}
\usepackage{tikz,tikz-cd}
\usepackage{hyperref}

\usetikzlibrary{matrix,calc,positioning,arrows,decorations.pathreplacing,patterns,arrows}

\newcommand{\cC}{\mathcal{C}}
\newcommand{\cE}{\mathcal{E}}

\newcommand{\cW}{\mathcal{W}}

\newcommand{\bD}{\mathbb{D}}

\newcommand{\bQ}{\mathbb{Q}}

\newcommand{\D}{\text{Diff}}
\newcommand{\Co}{\text{Cont}}
\newcommand{\AC}{\text{AlmCont}}

\newtheorem{theorem}{Theorem}[section]
\newtheorem*{thma}{Theorem A}
\newtheorem*{thmb}{Theorem B}
\newtheorem{lemma}[theorem]{Lemma}

\newtheorem{corollary}[theorem]{Corollary}
\newtheorem*{corollary*}{Corollary}

\theoremstyle{definition}
\newtheorem{definition}[theorem]{Definition}

\theoremstyle{remark}

\newtheorem{remark}[theorem]{Remark}

\newtheorem*{quest}{Question}

\title{Relative H-Principle and Contact Geometry}

\author{Jacob Taylor}
\address{Department of Mathematics, University of Toronto, 27 King's College Circle, Toronto, ON M5S 1A1, Canada}
\email{jacobw.taylor@mail.utoronto.ca}
\date{October 6, 2022}

\begin{document}
	
\begin{abstract} We show that if \(F(M)\) is some space of holonomic solutions with space of formal solutions \(F^f(M)\) that satisfies a certain relative \(h\)-principle, then the non-relative map \(F(M) \to F^f(M)\) admits a section up to homotopy. We apply this to the relative \(h\)-principle for overtwisted contact structures proved by Borman-Eliashberg-Murphy to find infinite cyclic subgroups in the homotopy groups of the contactomorphism group of \(M\).
\end{abstract}

\maketitle

\section{Introduction}
\label{sec:one}

In 1969, Gromov showed in \cite{Gromov} that if \(M\) is an open manifold then the inclusion \(\Co(M) \to \AC(M)\) is a weak equivalence, where \(\Co(M)\) is the space of contact structures on \(M\) and \(\AC(M)\) is the space of almost contact structures on \(M\). The case of closed manifolds is not so simple. For example, there exist contact structures on closed 3-manifolds that are formally homotopic but not homotopic, see \cite{Bennequin}. In \cite{Borman}, Matthew Borman, Yakov Eliashberg and Emmy Murphy advanced the field of contact geometry by first extending the definition of an overtwisted contact manifold from \(3\)-dimensional manifolds to all manifolds of dimension \(2n+1 \geq 3\), and then proving an \(h\)-principle result for overtwisted contact manifolds. Essentially, an overtwisted contact manifold is a contact manifold \(M\) that contains an embedded overtwisted disk, i.e. an embedded \(2n\)-disk \(\Delta\) with a certain model germ of a contact structure on a neighborhood of \(\Delta\) (see \cite{Borman} Definition 3.6). If \(\Co^{OT}(M,\Delta)\), \(\AC(M,\Delta)\) denote the spaces of contact and formal contact structures that are overtwisted with fixed disk \(\Delta\) respectively, then the main result of \cite{Borman} is that 
\[\Co^{OT}(M,\Delta) \to \AC(M,\Delta)\]
is a weak equivalence. However, it is known that in general the map \(\Co^{OT}(M) \to  \AC(M)\) from overtwisted contact structures to almost contact structures is not a weak equivalence, see for example \cite{Vogel}. From this, one may wonder how much can be known about the maps \(\Co^{OT}(M) \to \AC(M)\) and \(\Co(M) \to \AC(M)\) given that there is an \(h\)-principle when one fixes a disk? In fact, there is a much more general question here about relative \(h\)-principles, motivated by this example.

\begin{quest}
Let \(\Delta\) be some subset of \(M\) and \(\gamma\) be the germ of some holonomic solution on \(\Delta\).  Let \(F(M \text{ rel } (\Delta, \gamma))\) denote the set of all holonomic solutions that have germ \(\gamma\) on \(\Delta\), and \(F^f(M \text{ rel } (\Delta, \gamma))\) denote the set of formal solutions that have germ \(\gamma\) on \(\Delta\). If the map 
\[F(M \text{ rel } (\Delta, \gamma)) \to F^f(M \text{ rel } (\Delta, \gamma))\]
is a weak equivalence for all pairs \((\Delta,\gamma) \in \cW\) for some collection \(\cW\), what can be said about the map
\[F(M) \to F^f(M) \text{?}\]
\end{quest}

One of the main results of this paper is an answer to this question.
\noindent\phantomsection\label{Main}
\begin{thma}
Let \(\cW\) be a {\em sufficiently separated collection} (see Definition \ref{Sufficiently-Separated-Collection}). If the natural inclusion map \(F(M \text{ rel } (\Delta_1, \gamma_1), \dots, (\Delta_k,\gamma_k)) \to F^f(M \text{ rel } (\Delta_1,\gamma_1), \dots, (\Delta_k,\gamma_k))\) is a weak equivalence for all finite tuples of disjoint elements \(\{(\Delta_i,\gamma_i)\}_{i=1}^k \in \cW\), \(k \geq 1\), then there exists a space \(X\) and a map \(X \to F(M)\) such that the following diagram commutes:

\[\begin{tikzcd}
X \arrow[r] \arrow[rd, "\simeq"'] & F(M) \arrow[d] \\
                                 & F^f(M)        
\end{tikzcd}\]
where the map \(X \to F^f(M)\) is a weak equivalence. 
\end{thma}
In other words, the map from holonomic to formal solutions admits a section up to homotopy. This has some immediate consequences in contact geometry, as the above theorem allows us to find a subgroup of \(\pi_k\Co^{OT}(M)\) isomorphic to \(\pi_k\AC(M)\), induced by a map of spaces, for all \(k\). This is an improvement on the current tool used to analyze the difference between the homotopy groups of \(\Co^{OT}(M)\) and \(\AC(M)\), the {\em overtwisted group} (see \cite{Casals} Proposition 1 or \cite{Fernandez-Gironella} Definition \(10\)), which only allows one to realize \(\pi_k\AC(M)\) as a subgroup of \(\pi_k\Co^{OT}(M)\) when \(1 \leq k \leq 2n\). Furthermore, we show that these subgroups agree when \(1 \leq k \leq 2n\). Finally, we use these results to help study certain homotopy groups of the contactomorphism group of an overtwisted contact manifold. Let \(\cC_0(M,\xi_{OT})\) denote the identity component of the space of contactomorphisms of the contact manifold \((M,\xi_{OT})\).
\noindent\phantomsection\label{Second-Main}
\begin{thmb}
If \((M,\xi_{OT})\) is a closed, cooriented, overtwisted contact manifold of dimension \(2n+1\), then \(\pi_k\cC_0(M,\xi_{OT})\) contains an infinite cyclic subgroup whenever \(\pi_k\D_\partial(\bD^{2n+1}) \otimes \bQ \neq 0\), for \(k \leq  \phi^\bQ(\bD^{2n})-1\), \(k \neq 0\)
\end{thmb} Here \(\phi^\bQ(\bD^{2n})\) is the rational concordance stable range for \(\bD^{2n}\), see for example \cite{Goodwillie}.

\subsection*{Overview of the paper}

In \hyperref[sec:two]{Section Two}, we show that for a fiber bundle \(E \to M\) with space of continuous sections \(\Gamma(E)\), if \(\cW\) is a {\em sufficiently separated collection} of pairs \((\Delta,\gamma)\) for \(E\) (see Definition \ref{Sufficiently-Separated-Collection}) then the space \(\Gamma(E)\) is weak equivalent to a certain simplicial space built from the spaces of relative sections. Then in \hyperref[sec:three]{Section Three}, we use this result to prove Theorem \hyperref[Main]{A}. In \hyperref[sec:four]{Section Four}, we show that the collection of all overtwisted disks on a closed, connected manifold is sufficiently separated, and so the map \(\Co^{OT}(M) \to \text{AlmCont}(M)\) admits a section up to homotopy. Then we use this to show that in the range it is defined, the usual overtwisted group agrees with the image of the map induced by this section. In \hyperref[sec:five]{Section Five} we use these results to find infinite cyclic subgroups in the homotopy groups of the contactomorphism group of a closed overtwisted contact manifold, in degrees different than those found in \cite{Fernandez-Gironella}. Finally, in \hyperref[sec:six]{Section Six} we note other applications to Theorem \hyperref[Main]{A}, specifically coming from Engel geometry.  

\subsection*{Acknowledgments}
I would like to thank my supervisor Alexander Kupers for suggesting this topic and for his patience, support and guidance throughout this project. This work is partly supported by an Ontario Graduate Scholarship.

\section{Semisimplicial Resolutions of Section Spaces}
\label{sec:two}

For technical reasons, suppose we are in the category of compactly generated spaces. 
Let \(M\) be a \(d\)-dimensional manifold without boundary, and \(X\) be a path connected space. Suppose \(\pi: E \to M\) is a fiber bundle with fiber \(X\).

\begin{definition}\label{Germ}
Let \(A \subset M\) be some subset of \(M\). A {\em germ of a section on} \(A\) is a pair \((\gamma,U)\), where \(U\) is an open neighborhood of \(A\) and \(\gamma\) is a section on \(U\), with the equivalence relation that two germs are the same if they agree on some neighborhood of \(A\). 
\end{definition}
For convenience we usually omit the neighborhood \(U\) and just let \(\gamma\) denote the germ, with the understanding that \(\gamma\) is defined on some arbitarily small neighborhood of \(A\).  
\begin{definition}\label{Sufficiently-Separated-Collection}
A {\em sufficiently separated collection} \(\cW\) for the bundle \(\pi:E \to M\) is a collection of pairs \((\Delta,\gamma)\) of (a) a contractible compact subset \(\Delta \subset M\) and (b) a germ of a section of \(\pi\) near \(\Delta\). These are required to satisfy
\begin{enumerate}[(1)]
\item There exists some neighborhood \(D \cong \bD^d\) such that \(\Delta \subset int(D)\) and the inclusion map \(\iota: \Delta \to D\) is a closed cofibration. 
\item Given any finite collection \((\Delta_i,\gamma_i)_{i=1}^k\) in \(\cW\), there exists some \((\Delta',\gamma') \in \cW\) so that \(\Delta'\) is contained in the interior of a closed ball \(D \cong \bD^d\) that is disjoint from each \(\Delta_i\). 

\end{enumerate}
\end{definition}

We say that two elements \((\Delta_1,\gamma_1)\), \((\Delta_2,\gamma_2)\) of \(\cW\) are {\em disjoint} if \(\Delta_1 \cap \Delta_2 = \emptyset\), and a collection of elements is {\em disjoint} if each pair of elements in the collection are disjoint. For a pair \((\Delta, \gamma)\) in such a collection, let \(\Gamma_g(E, \Delta,\gamma)\) denote the space of sections on \(E\) that have germ \(\gamma\) near \(\Delta\), and let \(\Gamma(E,\Delta,\gamma)\) denote the space of sections \(f\) such that \(f|_\Delta = \gamma|_\Delta\). For convenience, we omit \(\gamma\) when there is no confusion and write \(\Gamma_g(E,\Delta)\) and \(\Gamma(E,\Delta)\) for \(\Gamma_g(E, \Delta,\gamma)\) and \(\Gamma(E,\Delta,\gamma)\) respectively. By definition, \(\Gamma_g(E,\Delta)\) is topologized as the colimit 
\[\Gamma_g(E,\Delta) := \lim_\to \Gamma(E,B_i),\]
Where the \(B_i\) are some neighborhood basis of \(\Delta\). 
Then we have the following lemma, which will allow us to forget about germs and just work with spaces of sections that have fixed values on subsets.
\begin{lemma}\label{Germs-Are-Irrelevant}
The map \(\Gamma_g(E, \Delta) \to \Gamma(E,\Delta) \) given by inclusion is a weak equivalence. 
\end{lemma}
\begin{proof}
To show this, one can show that given the following diagram
\[\begin{tikzcd}
S^{n-1} \arrow[r] \arrow[d] & {\Gamma_g(E,\Delta)} \arrow[d] \\
\mathbb{D}^n \arrow[r]      & {\Gamma(E,\Delta)}            
\end{tikzcd}\]
there exists a lift \(\beta:\bD^n \to \Gamma_g(E,\Delta)\) of \(\alpha\) up to homotopy relative to the boundary.
By Lemma \(3.6\) of \cite{Strickland}, any map from a compact space to a colimit of closed inclusions factors over one of the inclusions, so \(S^{n-1} \to \Gamma_g(E,\Delta)\) factors as \(S^{n-1} \to \Gamma(E,B_i) \to \Gamma_g(E,\Delta)\) for some neighborhood \(B_i\). We can choose \(B_i\) to be contractible and such that the inclusion \(B_i \to M\) is a closed cofibration. Then it suffices to show that  \(\Gamma(E,B_i) \to \Gamma(E,\Delta)\) is a weak equivalence, which holds due to the following commutative diagram; The rows are fiber sequences and \(\Gamma(E|_{B_i}) \to \Gamma(E|_{\Delta})\) is a weak equivalence since both \(\Delta\) and \(B_i\) are contractible.
\[\begin{tikzcd}
{\Gamma(E,B_i)} \arrow[d, "\iota"'] \arrow[r] & \Gamma(E) \arrow[d, "id"'] \arrow[r] & \Gamma(E|_{B_i}) \arrow[d, "res"'] \\
{\Gamma(E,\Delta)} \arrow[r]                  & \Gamma(E) \arrow[r]                  & \Gamma(E|_{\Delta})               
\end{tikzcd},
\]

\end{proof}
\begin{remark}
A similar argument can be used if one replaces \((\Delta,\gamma)\) with finitely many disjoint elements of \(\cW\).
\end{remark}
Next, we can construct a semi-simplicial space from \(\cW\), \(Y_\bullet\), by letting 
\[Y_p := \coprod_{\substack{p+1 \text{ disjoint elements of } \cW (\Delta_0,\gamma_0), \dots , (\Delta_p,\gamma_p)}} \Gamma(E, \Delta_0, \dots, \Delta_p) ,\]
where the face maps are given by forgetting subspaces. Here \(\Gamma(E, \Delta_0, \dots, \Delta_p)\) is the space of sections \(f\) of \(E\) that satisfy \(f|_{\Delta_i} = \gamma_i\). 
We can also consider the semi-simplicial space \(W_\bullet\) given by
\[W_p := \coprod_{p+1 \text{ disjoint elements of } \cW (\Delta_0,\gamma_0), \dots , (\Delta_p,\gamma_p)},\]
where again the face maps are given by forgetting subspaces. 
These semi-simplicial spaces are related via the following lemma. 

\begin{lemma}\label{Geometric-Realization-Explicit-Formulation} The space \(||Y_\bullet||\) is homeomorphic to the subspace of \(\Gamma(E) \times || W_\bullet ||\) given by \(\{ (f,\overrightarrow{w},\overrightarrow{t}) | f(m) = \gamma_i(m) \text{ whenever } m \in \Delta_i \text{ and } t_i \neq 0\}/ \sim\), where \((\Delta_i,\gamma_i)\) is the ith component of \(\overrightarrow{w}\), and \(\sim\) is just the usual geometric realization equivalence on the second factor. 
\end{lemma}
\begin{proof}
It is clear these are the same as sets, so we just need to show that this map is a homeomorphism onto some subspace. 
First, since the quotient of a subspace is naturally a subspace of the quotient in compactly generated spaces, we have 
\[||Y_\bullet|| \subset ||\Gamma(E) \times W_\bullet||. \]
On the other hand, by \cite{Randal-Williams} page \(2106\) we have that 
\[||\Gamma(E) \times W_\bullet|| = ||\Gamma(E) \otimes W_\bullet|| \cong ||\Gamma(E)|| \times ||W_\bullet|| \cong \Gamma(E) \times ||W_\bullet||,\]
where we are treating \(\Gamma(E)\) as a semi-simplicial space with only \(0\)-simplices in order to use the exterior product defined in \cite{Randal-Williams} page \(2103\).
So, \(||Y_\bullet||\) is homeomorphic to the subspace of \(\Gamma(E) \times ||W_\bullet||\) described above. 
\end{proof} 

We can now introduce the main result of this section:

\begin{theorem}\label{Section-Space-Weak-Equivalence} Let \(\cW\) be a sufficiently separated collection of subsets of \(M\), and \(||Y_\bullet||\) be as above. The map \(||Y_\bullet|| \to \Gamma(E)\) given by forgetting the fixed subsets is a weak equivalence. 
\end{theorem}
\begin{proof}
We will prove this by showing that the relative homotopy groups of the map are zero. So, let \(\alpha:\bD^n \to \Gamma(E)\) be a continuous map such that we have the following diagram:

\[\begin{tikzcd}
S^{n-1} \arrow[r, "\partial \alpha"] \arrow[d] & ||||Y_\bullet|| \arrow[d] \\
\bD^n \arrow[r, "\alpha"]                        & {\Gamma(E)}             
\end{tikzcd}\]
where by abuse of notation \(\partial \alpha\) is the map \(S^{n-1} \to ||Y_\bullet||\), \( p \mapsto (\alpha(p), \overrightarrow{w_p}, \overrightarrow{t_p})\) for some finite ordered set of elements \(\overrightarrow{w_p}\) of \(\cW\) and weights \(\overrightarrow{t_p}\) (where of course we just exclude any elements when their weight goes to zero). Then we need to show that there exists a continuous map \(
\alpha' \simeq \alpha\) relative to the boundary, and a lift \(\beta: \bD^n \to ||Y_\bullet||\) of \(\alpha'\) such that the resulting diagram still commutes.

First, consider the section 
\[\alpha_1(p)(m) := \begin{cases} \alpha(2p)(m) & 0 \leq |p| \leq \frac{1}{2} \\
						\alpha(\frac{p}{|p|}) & \frac{1}{2} \leq |p| \leq 1
\end{cases}.\]
So \(\alpha_1\) is just \(\alpha\) compressed to a smaller ball, with the boundary extended to an annulus. It is clear that \(\alpha_1 \simeq \alpha\) relative to the boundary, so we can work with \(\alpha_1\), which gives us a buffer away from the boundary. Next, consider the set \(\tilde{W} := \{ m \in M |\) for some \(p \in S^{n-1}\) and some positive integer \( i, m \in \Delta_{p,i}\), where \( (\Delta_{p,i},\gamma_{p,i})\) is the ith component of \( \overrightarrow{w}_p\}\). Note that \(\alpha_1\) lifts to \(||Y_\bullet||\) on the annulus, and furthermore if we only change \(\alpha_1\) away from \(\tilde{W}\) then it will still lift in the annulus. Now, using the natural projection map from \(||Y_\bullet|| \to ||W_\bullet||\) given in Lemma \ref{Geometric-Realization-Explicit-Formulation}, we can consider the map 
\[S^{n-1} \to ||Y_\bullet|| \to ||W_\bullet||,\]
where \(p \mapsto (\overrightarrow{w_p}, \overrightarrow{t_p})\). Since \(S^{n-1}\) is compact we know that the image of this map is compact, and so it hits finitely many cells of \(||W_\bullet||\). Also, each cell consists of finitely many elements of \(\cW\), so the set of all elements of \(\cW\) that are a component of \(\overrightarrow{w_p}\) for some \(p\) is finite. But, if \(\{(\Delta_1,\gamma_1),\dots, (\Delta_k,\gamma_k)\}\) is that set, then clearly \(\tilde{W} = \Delta_1 \cup \dots \cup \Delta_k\). Let \((\Delta,\gamma) \in \cW\) be such that there is some neighborhood \(D \cong \bD^d\) of \(\Delta\) that is disjoint from \(\tilde{W}\). Since a map \(\bD^n \to \Gamma(E)\) is the same data as a section of the bundle \(\bD^n \times E \to \bD^n \times M\), \((p,e) \mapsto (p,\pi(e))\) we will from now on consider \(\alpha, \alpha_1\) as maps from \(\bD^n \times M \to \bD^n \times E\) so that composition with the projection map is the identity. Now, let \(g_0 = \alpha_1|_{\bD^n \times D}\), \(\bD^n_{\frac{3}{4}}:= \{ p \in \bD^n | |p| \leq \frac{3}{4}\}\). Since \(\bD^n \times D\) is a contractible submanifold of \(\bD^n \times M\), we know that the restriction of the bundle to \(\bD^n \times D\) is trivial, and so there is some trivialization of the bundle \(\bD^n \times M = \cup_i U_i\) such that \(\bD^n \times D\) is completely contained in some \(U_i\). That is, \(g_0\) and any other sections on (a subset of) \(\bD^n \times D\) just become maps from (a subset of) \(\bD^n \times D\) to \(X\).  Consider the homotopy 
\[(\partial \bD^n \cup \bD^n_{\frac{3}{4}}) \times (\partial D \cup \Delta) \times [0,1] \to X\]
given by 
\[H(p,m,s) := \begin{cases} \alpha_1(p,m) & p \in \partial \bD^n \\
					\alpha_1(p,m) & m \in \partial D \\
					h(p,m,s) & (p,m) \in \bD^n_{\frac{3}{4}} \times \Delta
\end{cases}, \]
Where \(h(p,m,s)\) is a homotopy between \(\alpha_1|_{\bD^n_{\frac{3}{4}} \times \Delta}\) and the map \((p,m) \mapsto \gamma(m)\). Such an \(h\) exists since \(\bD^n_{\frac{3}{4}} \times \Delta\) is contractible and \(X\) is path connected, which implies any two maps are homotopic. Since the inclusion of \(\Delta\) into \(D\) is a cofibration, and the boundary of \(D\) is disjoint from \(\Delta\), we have that \((D,\partial D \cup \Delta)\) satisfies the homotopy extension property. Clearly \((\bD^n, \partial \bD^n \cup \bD_{\frac{3}{4}}^n)\) does as well, so we can apply the homotopy extension property twice, once for each of the two factors, to get a homotopy 
\[\bD^n \times D \times [0,1] \to X\] 
between \(g_0\) and a function \(g_1:\bD^n \times D \to X\), relative to the boundary, so that \(g_1 = g_0\) on \(\partial \bD^n\), \(\partial D\), and \((p,m) \mapsto \gamma(m)\) on \( \bD^n_{\frac{3}{4}} \times \Delta\). Now, we can extend this to a section on all of \(M\) by letting 
\[\alpha_2(p,m) := \begin{cases} \alpha_1(p,m) & m \notin D \\
						g_1(p,m) & m \in D
\end{cases}.\]
Clearly \(\alpha_2 \simeq \alpha_1\) relative to the boundary, \(\alpha_2(p,m) = \gamma(m)\) when \(|p| \leq \frac{3}{4}\), \(m \in \Delta\), and for \(m \in \tilde{W}\) we have \(\alpha_2(p,m) = \alpha_1(p,m) = \gamma_{\frac{p}{|p|},i}(m)\) when \(|p| \geq \frac{1}{2}\). So, all we need to do now is show that \(\alpha_2\) lifts. Let us again view these as maps from \(\bD^n \to \Gamma(E)\), and consider the map \(\beta: \bD^n \to ||Y_\bullet||\) given by 
\[\beta(p) := \begin{cases} (\alpha_2(p),(\Delta,\gamma),1) & 0 \leq |p| \leq \frac{1}{2} \\
					(\alpha_2(p), (\overrightarrow{w}_{\frac{p}{|p|}},(\Delta,\gamma)), ((4|p|-2)\overrightarrow{t}_{\frac{p}{|p|}}, 3-4|p|)) & \frac{1}{2} < |p| \leq \frac{3}{4} \\
					(\alpha_2(p), \overrightarrow{w}_{\frac{p}{|p|}},\overrightarrow{t}_{\frac{p}{|p|}}) & \frac{3}{4} < |p| \leq 1
\end{cases}.\]
Clearly \(\beta\) is a lift of \(\alpha_2\), as required.
\end{proof}
\begin{remark}
If we let \(A \subset M\) be a closed set such that \(M \setminus A\) is connected, and \(\xi_0\) be a section on a neighborhood of \(A\), we can run the exact same argument if we suppose \(\cW\) is a sufficiently separated collection in \(M \setminus A\) and consider sections on \(M\) that agree with \(\xi_0\) near \(A\).
\end{remark}
\section{H-Principle}
\label{sec:three}

Let \(M\) be a \(d\)-dimensional manifold without boundary. We would like to use the previous theorem to show that the map \(F(M) \to F^f(M)\) admits a section up to homotopy under the right conditions. We know there exists some bundle \(E \to M\) such that \(F^f(M)\) is the space \(\Gamma(E)\) of sections of \(E\). Suppose furthermore that the fibers of \(E\) are path connected, and let \(W\) be a sufficiently separated collection for the bundle \(E \to M\). 
\begin{theorem}\label{Relative-H-Principle}
If the natural inclusion map \(F(M \text{ rel } (\Delta_1, \gamma_1), \dots, (\Delta_k,\gamma_k)) \to F^f(M \text{ rel } (\Delta_1,\gamma_1), \dots, (\Delta_k,\gamma_k))\) is a weak equivalence for all finite tuples of disjoint elements \(\{(\Delta_i,\gamma_i)\}_{i=1}^k \in \cW\), \(k \geq 1\), then there exists a space \(X\) and a map \(X \to F(M)\) such that the following diagram commutes:

\[\begin{tikzcd}
X \arrow[r] \arrow[rd, "\simeq"'] & F(M) \arrow[d] \\
                                 & F^f(M)        
\end{tikzcd}\]
where the map \(X \to F^f(M)\) is a weak equivalence. 
\end{theorem}
\begin{proof}
First, let \(X_\bullet\) be the semi-simplicial space given by 
\[X_p := \coprod_{p+1 \text{ disjoint elements of } \cW (\Delta_0,\gamma_0), \dots , (\Delta_p,\gamma_p)} F(M \text{ rel } \Delta_0, \dots, \Delta_p),\]
and \(Y_\bullet\) be the semi simplicial space given by replacing each \(F(M)\) in \(X_\bullet\) with its formal version, i.e.
\[Y_p := \coprod_{p+1 \text{ disjoint elements of } \cW (\Delta_0,\gamma_0), \dots , (\Delta_p,\gamma_p)} F^f(M \text{ rel } \Delta_0, \dots, \Delta_p).\]
By assumption, \(F(M \text{ rel } \Delta_0, \dots, \Delta_p) \to F^f(M \text{ rel } \Delta_0, \dots, \Delta_p)\) is a weak equivalence for all tuples of elements of \(\cW\), and so \(||X_\bullet|| \to ||Y_\bullet||\) is also a weak equivalence. Also, by Theorem \ref{Section-Space-Weak-Equivalence}, \(||Y_\bullet|| \to  F^f(M)\) is a weak equivalence, so we get the following commutative diagram:
\[\begin{tikzcd}
||||X_\bullet|| \arrow[r] \arrow[d, "\simeq"] & F(M) \arrow[d] \\
||||Y_\bullet|| \arrow[r, "\simeq"]           & F^f(M)        
\end{tikzcd}
\]
where the vertical maps are induced by the inclusion \(F(M) \to F^f(M)\) and the horizontal maps come from forgetting about the fixed sets. 
\end{proof}
\begin{remark}
A more common formulation of such a relative \(h\)-principle result is that there is an \(h\)-principle relative to any fixed closed set \(A\) and fixed subset \(\Delta\), where \(\Delta\) is from some special collection. Such a result will usually imply the result for many fixed subsets \(\Delta_1, \dots, \Delta_k\), since we can take \(A\) to be the union of the first \(k-1\) such sets, and use \(\Delta_k\) as our fixed subset \(\Delta\). 
\end{remark}
\section{Improved H-Principle for Contact Geometry}
\label{sec:four}

Let us briefly recall some basic definitions from contact geometry. A (cooriented) contact structure on a connected, orientable, \(2n+1\)-dimensional manifold \(M\) is a ``maximally non-integrable'' hyperplane distribution \(\xi = ker(\alpha)\) for a \(1\)-form \(\alpha\). Maximally non-integrable means \(\alpha \wedge d\alpha^n \neq 0\). This naturally induces a reduction of the structure group of \(M\) to \(U(n) \times 1\), and so an {\em almost contact structure} is just a reduction of the structure group of \(M\) to \(U(n) \times 1\). Equivalently, an almost contact structure is a triple \((\xi,J,R)\), where \(\xi\) is a hyperplane distribution, \(J\) is a complex structure on \(\xi\), and \(R\) is a trivial sub - line bundle of \(TM\) such that \(\xi \oplus R = TM\). We let \(\Co(M)\) denote the space of contact structures on \(M\) and \(\AC(M)\) denote the space of almost contact structures on \(M\). 

Next, we recall the notion of an overtwisted contact structure. An {\em overtwisted disk} in a manifold \(M\) is a pair \((\Delta,\gamma)\), where \(\Delta \subset M\) is an embedded \(2n\)-dimensional disk and \(\gamma\) is a certain model germ of a contact structure on \(\Delta\). Then a contact manifold \((M,\xi)\) is said to be overtwisted if there exists an embedding of an overtwisted disk \((\Delta,\gamma)\) such that the contact germ \(\gamma\) agrees with \(\xi\) on some neighborhood of \(\Delta\) (i.e. the embedding is a contact embedding). In this case we say that \(\Delta\) is overtwisted for \(\xi\). The details of this definition in any dimension can be found in \cite{Borman} (Definition 3.6). Let \(\Co^{OT}(M)\) denote the space of overtwisted contact structures on \(M\).
It has been shown by \cite{Borman} (Theorem \(1.2\)) that if \(M\) is a closed \(2n+1\)-dimensional manifold, \(A\) is a closed subset of \(M\) such that \(M \setminus A\) is connected, \((\Delta,\gamma)\) is an overtwisted disk in \(M \setminus A\), and \(\xi_0\) is an almost contact structure on \(M\) that is a genuine contact structure on a neighborhood of \(A\), then the map 
\[\Co^{OT}(M,(A,\xi_0),(\Delta,\gamma)) \to \text{AlmCont}(M,(A,\xi_0),(\Delta,\gamma))\]
is a weak equivalence. Here \(\Co^{OT}(M,(A,\xi_0),(\Delta,\gamma))\) is the space of contact structures on \(M\) that agree with \(\xi_0\) in a neigborhood of \(A\) and are overtwisted with disk \(\Delta\), and \(\text{AlmCont}(M,(A,\xi_0),(\Delta,\gamma))\) is the corresponding space of almost contact structures.

From this, we would like to prove the following theorem: 
\begin{theorem}\label{Improved-Contact-Geometry}
Let \(M\) be a closed, connected, \(2n+1\) dimensional manifold. Then the map 
\[\Co^{OT}(M) \to \text{AlmCont}(M),\] 
admits a section up to homotopy.  
\end{theorem}

\begin{proof} 
 First, let \(X_\bullet\) be given by 
\[X_p := \coprod_{p+1 \text{ disjoint overtwisted disks } (\Delta_0,\gamma_0), \dots , (\Delta_p,\gamma_p)} \Co^{OT}(M \text{ rel } \Delta_0, \dots, \Delta_p),\]
and \(Y_\bullet\) be given by 
\[Y_p := \coprod_{p+1 \text{ disjoint overtiwsted disks } (\Delta_0,\gamma_0), \dots , (\Delta_p,\gamma_p)} \AC(M \text{ rel } \Delta_0, \dots, \Delta_p).\]
It is known (see \cite{Fernandez-Gironella} page 191) that \(\text{AlmCont}(M)\) is naturally the section space of a certain fiber bundle on \(M\) with connected fiber \(SO(2n+1)/U(n)\), which comes from viewing \(\AC(M)\) as the space of reductions of the structure group of \(M\) to \(U(n) \times 1\). Also, if we have some collection of disjoint overtwisted disks \(\Delta_1, \dots, \Delta_k\), we can let \(A = \Delta_1 \cup \dots \cup \Delta_{k-1}\) and \(\Delta_k\) be the overtwisted disk, so the \(h\)-principle above implies an \(h\)-priniciple of the form required in Theorem \ref{Relative-H-Principle}. So, we just need to show that the collection of all overtwisted disks on \(M\) is sufficiently separated. Clearly such an embedded disk is closed, contractible and compact. Also, by finding a regular neighborhood, it is clear that any embedded \(2n\) - dimensional disk is a neighborhood deformation retract of a \(2n+1\) - dimensional ball containing it, so the inclusion is a cofibration and condition \((1)\) of Definition \ref{Sufficiently-Separated-Collection} is satisfied. Finally, it is clear that a finite collection of embedded overtwisted disks in \(M\) can't cover \(M\), so given some finite set of overtwisted disks in \(M\) there will always be some point \(m \in M\) that is not in any of them, and since an overtwisted disk is just some embedded disk with a local germ, we can always introduce a new overtwisted disk at the point \(m\) that doesn't intersect the rest of the overtwisted disks, and condition \((2)\) of Definition \ref{Sufficiently-Separated-Collection} is satisfied.
\end{proof}
\begin{remark}
This argument also shows that the map \(\Co(M) \to \AC(M)\) admits a section up to homotopy, but this section factors through the overtwisted contact structures. 
\end{remark}
\begin{remark}
Similar results should hold for manifolds with boundary, using that the \(h\)-principle given in \cite{Borman} works relative to any closed set \(A\) as long as \(M \setminus A\) is connected. 
\end{remark}

From Theorem \ref{Improved-Contact-Geometry} we get as an immediate consequence that \(\pi_k\text{AlmCont}(M)\) is isomorphic to a subgroup of \(\pi_k\Co^{OT}(M)\) for all \(k\). This is an improvement on the current overtwisted group \(OT_k(M)\) (see \cite{Casals} Proposition \(A.2\)), which gives an isomorphism between \(\pi_k\text{AlmCont}(M)\) and a subgroup of \(\pi_k \Co^{OT}(M)\) when \(1 \leq k \leq 2n\). In fact, when \(1 \leq k \leq 2n\) the image of \(\pi_k||X_\bullet||\) in \(\pi_k\Co^{OT}(M)\) is \(OT_k(M)\):
\begin{theorem}\label{Overtwisted-Group}
The overtwisted group \(OT_k(M)\) is the image of \(\pi_k||X_\bullet||\) induced by the natural forgetful map \(||X_\bullet|| \to \Co^{OT}(M)\) when \(1 \leq k \leq 2n\).
\end{theorem}
Before we can prove this, we need to define some intermediary spaces that will help us understand the relationship between \(\pi_k||X_\bullet||\) and \(OT_k(M)\). Recall that an element of \(||X_\bullet||\) is a contact structure, along with a list of disks and weights, such that each disk is overtwisted for the contact structure as long as its weight is nonzero. Since we defined this for an arbitrary space of holonomic solutions and an arbitrary sufficiently separated collection, we didn't make use of any topology on the space of disks. So, in \(||X_\bullet||\) the fixed disks aren't allowed to move through \(M\), and the only way to change the disks is to introduce a new overtwisted disk somewhere, or delete a disk by letting its weight go to zero. However, the space of overtwisted disks does have a topology coming from the space of embeddings of \(\bD^{2n}\) into \(M\). We can use this to define new semi-simplicial spaces that are more clearly related to \(OT_k(M)\). 
\begin{remark}
Strictly speaking, overtwisted disks are only piecewise smooth, so instead of embeddings of the standard disk into \(M\) we want to take a specific piecewise structure coming from the model overtwisted disk (see \cite{Borman} Definition \(3.6\)) and consider the space of embeddings of this into \(M\) that preserve the piecewise smooth structure. However, none of our arguments depend on this distinction, so by abuse of notation we just denote this space as \(\text{Emb}(\bD^{2n},M)\).
\end{remark}
\begin{definition}
Let \(X_\bullet^c\) be the semi-simplicial space defined by 
\[X_p^c \subset  \Co^{OT}(M) \times \text{Emb}(\bD^{2n},M)^{p+1}\] 
is the set of all \((\xi,\Delta_0,\dots,\Delta_p)\) such that each \(\Delta_i\) is an overtwisted disk for \(\xi\), \(\Delta_i, \Delta_j\) are disjoint when \(i \neq j\). The face maps \(d_i\) are given by forgetting the \(i\)th disk. 
\end{definition}

Similarly, let \(Y_\bullet^c\) be the same definition except with \(\AC(M)\) instead of \(\Co^{OT}(M)\). The geometric realizations of these differ from \(||X_\bullet||\) and \(||Y_\bullet||\) since continuous maps into \(||X_\bullet^c||\) and \(||Y_\bullet^c||\) can also deform disks along families inside of \(M\). For example, if \(\alpha:S^k \to \Co^{OT}(M)\) is a family of contact structures, and \(\Delta: S^k \to \text{Emb}(\bD^{2n},M)\) is a certificate of overtwistedness for \(\alpha\), then \((\alpha,\Delta,1)\) is naturally a continuous map \(S^k \to ||X_\bullet^c||\). Furthermore, since constant embeddings are still continuous embeddings, there are natural maps \(||X_\bullet|| \to ||X_\bullet^c||\) and \(||Y_\bullet|| \to ||Y_\bullet^c||\) given by viewing the fixed disks as constant embeddings. Then we have the following commutative diagram. 

\[\begin{tikzcd}
||||X_\bullet|| \arrow[d] \arrow[r] & ||||X_\bullet^c|| \arrow[d] \arrow[r] & \text{Cont}^{OT}(M) \arrow[d] \\
||||Y_\bullet|| \arrow[r]           & ||||Y_\bullet^c|| \arrow[r]           & \text{AlmCont(M)}       
\end{tikzcd}\]

We already know the map \(||X_\bullet|| \to ||Y_\bullet||\) is a weak equivalence by the \(h\)-principle given in \cite{Borman}. Also, we have the following lemma, which is a direct consequence of \cite{Borman}.
\begin{lemma}\label{Foliated_h_Principle}
The map \(||X_\bullet^c|| \to ||Y_\bullet^c||\) induced by \(\Co^{OT}(M) \to \AC(M)\) is a weak equivalence.
\end{lemma}
\begin{proof}
First, \(X_p^c \to Y_p^c\) is a weak equivalence as a direct consequence of Theorem 1.6 of \cite{Borman}. Indeed, suppose we have the following diagram: 
\[\begin{tikzcd}
S^{k-1} \arrow[d] \arrow[r, "\partial \alpha"] & X_p^c \arrow[d] \\
\mathbb{D}^k \arrow[r, "\alpha"]               & Y_p^c          
\end{tikzcd}\]
where \(\alpha(t) = (\xi(t),\Delta_0(t),\dots,\Delta_p(t))\). First, we can homotope \(\alpha\) relative to the boundary by extending the boundary to an annulus, so that \(\alpha\) is genuine on a neigborhood of the boundary. Then we can consider \(V = M \times \bD^k\), so that \(\xi\) can be viewed as a leafwise contact structure on \(V\). If we let \(A = S^{k-1} \times M \subset V\), \(\xi_0 = \xi|_{A}\), and \(h_i = \Delta_i\) for \(0 \leq i \leq p\), we have that 
\[\xi \in \AC(V;A,\xi_0,h_0,\dots,h_p),\]
where \(\AC(V;A,\xi_0,h_0,\dots,h_p)\) is the space of leafwise almost contact structures that agree with \(\xi_0\) near \(A\), with overtwisted basis \(\{h_i\}_{i=0}^p\) (see \cite{Borman} Theorem 1.6 and definitions immediately preceding). Then \(\xi\) is a representative of an element in \(\pi_0\AC(V;A,\xi_0,h_0,\dots,h_p)\).
But, if \(\Co(V;A,\xi_0,h_0,\dots,h_p)\) is the space of  leafwise contact structures that agree with \(\xi_0\) near \(A\), with overtwisted basis \(\{h_i\}_{i=0}^p\), then by \cite{Borman} Theorem 1.6 we have that 
\[\pi_0 \Co(V;A,\xi_0,h_0,\dots,h_p) \to \pi_0\AC(V;A,\xi_0,h_0,\dots,h_p)\]
is an isomorphism, so there is some \(\tilde{\xi} \in \Co(V;A,\xi_0,h_0,\dots,h_p)\) and a path from \(\xi\) to \(\tilde{\xi}\) in \(\AC(V;A,\xi_0,h_0,\dots,h_p)\). But, a path in this space is a homotopy of \(\xi\) relative to the boundary and relative to the families of overtwisted disks \(\Delta_0,\dots,\Delta_p\). Clearly such a homotopy is a homotopy of \(\alpha:\bD^k \to Y_p^c\) relative to the boundary to a map \(\beta = (\tilde{\xi},\Delta_0,\dots,\Delta_p)\), which lifts to \(X_p^c\). So indeed \(X_p^c \to Y_p^c\) is a weak equivalence and so \(||X_\bullet^c|| \to ||Y_\bullet^c||\) is also a weak equivalence. 
\end{proof}

Furthermore, we have the following which relates \(||Y_\bullet^c||\) to \(\AC(M)\):
\begin{lemma}\label{Continuous-Section-Space-Weak-Equivalence}
The map \(||Y_\bullet^c|| \to \AC(M)\) is a weak equivalence.
\end{lemma}
\begin{proof}
The proof is similar to the proof of Theorem \ref{Section-Space-Weak-Equivalence}, so we will omit some technical details that were included there. Let \(\alpha:\bD^k \to \AC(M)\) be a continuous map such that we have the following diagram
\[\begin{tikzcd}
S^{k-1} \arrow[r, "\partial \alpha"] \arrow[d] & ||||Y_\bullet^c|| \arrow[d] \\
\bD^k \arrow[r, "\alpha"]                        & {\AC(M)}             
\end{tikzcd}\]
where by abuse of notation \(\partial \alpha\) is the map \(S^{k-1} \to ||Y_\bullet^c||\), \( p \mapsto (\alpha(p), \overrightarrow{w_p}, \overrightarrow{t_p})\), where \(\overrightarrow{w_p} = (w_1(p), \dots w_\ell(p))\) is a finite ordered set of overtwisted disks for \(\alpha(p)\) and \(\overrightarrow{t_p} = (t_1(p), \dots, t_\ell(p))\) is their corresponding weights. We can assume all of these weights appearing are nonzero. There is only one part of the proof of \ref{Section-Space-Weak-Equivalence} that doesn't go through immediately, which is finding a disk that is disjoint from all the disks \(w_i(p)\) for all \(p,i\). The problem is that since the embedded disks are no longer locally constant but rather can vary from point to point, we have \(S^{k-1}\) families of disks instead of finitely many, so it is possible that they cover all of \(M\). However, we can get around this as follows. First, we can make a buffer away from the boundary by replacing \(\alpha\) with a radial compression to the disk of radius \(\frac{1}{2}\), which is homotopic to \(\alpha \) relative to the boundary. Again by abuse of notation we will let \(\alpha\) denote this new map. Let \(A\) be the annulus of radius \(\frac{1}{2}\), so now the map \(\alpha\) lifts to \(||Y_\bullet^c||\) in \(A\). Let \(S^{k-1}_{\frac{1}{2}}\) be the sphere of radius \(\frac{1}{2}\), and for each \(p \in S^{k-1}_{\frac{1}{2}}\) let \(V_p\) be a sufficiently small open neighborhood of \(p \in \bD^k\) such that there is some \(m \in M\) so that \(m\) is not contained in any \(w_i(q)\) for any \(q \in V_p \cap A\) or \(1 \leq i \leq \ell(p)\). Furthermore, we can let \(U_p \subset V_p\) be a slightly smaller neighorhood of \(p\) that is buffered from \(\partial V_p\) by some \(\epsilon\) neighborhood. Since spheres are compact, we can find a finite subcover by these smaller neighborhoods, \(U_1, \dots, U_j\), which also gives us a cover by the larger neighborhoods \(V_1,\dots, V_j\). By construction the disks \(w_i(q)\) for \(q \in V_a \cap A\) don't cover \(M\) for any given \(1 \leq a \leq j\), so in particular we can find an embedded disk \(\Delta_1 \subset M\) and a regular neighborhood \(D_1\) of \(\Delta_1\) in \(M\), such that \(D_1\) is disjoint from all such disks \(w_i(q)\), \(q\in V_1 \cap A\). Finally, we can pick some \(V_0 \subset int(\bD^k_{\frac{1}{2}})\) and slightly smaller \(U_0\) so that \(U_1, \dots, U_j,U_0,A\) cover \(\bD^k\). With all of this set up, we can now do the following. \(U_1 \times \Delta_1\) is contractible, so if we restrict \(\alpha \) to this we can homotope it to agree with the overtwisted germ that comes with \(\Delta_1\). Also, we can use the homotopy extension property to extend this homotopy to one on \(V_1 \times D_1\), such that on \(\partial D_1\), \(\partial V_1\) the homotopy is just \(\alpha\). Then we can extend the homotopy by \(\alpha\) to all of \(\bD^k\), \(M\) so that we have a new map \(\alpha_1: \bD^k \to \AC(M)\) that is homotopic to \(\alpha\), agrees with \(\alpha\) away from \(V_1 \times D_1\), and satisfies that \(\Delta_1\) is overtwisted for \(\alpha_1(p)\) for all \(p \in U_1\). Also, \(w_i(q)\) is still overtwisted for \(\alpha_1(q)\) for all \(q \in A\), since in \(V_1\), \(\Delta_1\) is away from all of the \(w_i(q)\), and outside of \(V_1\), \(\alpha_1 = \alpha\). We can repeat this on \(U_2\), now being careful to choose \(\Delta_2,D_2\) so that \(D_2\) is disjoint from \(\Delta_1\) as well as \(w_i(q)\) for \(q \in V_2 \cap A\). Then we can find \(\alpha_2\) homotopic to \(\alpha_1\) so that \(\alpha_2 = \alpha_1\) outside of \(V_2 \times D_2\) and \(\Delta_2\) is overtwisted for \(\alpha_2(p)\) for all \(p \in U_2\). So, by the same reasoning \(\Delta_1\) is still overtwisted for \(\alpha_2(p)\) for \(p \in U_1\) and \(w_i(q)\) is still overtwisted for \(\alpha_2(q)\) for all \(q \in A\). Repeat this for \(U_3,\dots,U_j\), and get \(\alpha_j\), which still has the same overtwisted disks in \(A\) as well as \(\Delta_a\) is overtwisted for \(\alpha_j\) on \(U_a\). Finally, since \(V_0\) is disjoint from the annulus \(A\) by definition, on \(U_0\) we just need to find an overtwisted disk \(\Delta_0\) with regular neighborhood \(D_0\) disjoint from \(\Delta_1, \dots, \Delta_j\), which we can always do. Then do the same homotopy trick as before to get \(\alpha_0\) that agrees with \(\alpha_j\) outside of \( V_0 \times D_0\) and has overtwisted disk \(\Delta_0\) on \(U_0\). Finally, let \(s_0,\dots, s_j,s_A\) be a partition of unity subordinate to the cover \(U_0,U_1, \dots, U_j, A\). Then we have a map \(\beta: \bD^k \to ||Y_\bullet^c||\) given by 
\[p \mapsto (\alpha_0(p), \Delta_0,\Delta_1,\dots,\Delta_j,\overrightarrow{w_\frac{p}{|p|}}, s_0(p), s_1(p), \dots, s_j(p), s_A(p)\overrightarrow{t_{\frac{p}{|p|}}}).\]
By construction the specified disk is an overtwisted disk at \(p\) precisely when its weight is nonzero, all of the weights add up to one, and this is clearly a lift of \(\alpha_0\) which is homotopic to \(\alpha\) relative to the boundary. 
\end{proof}
Using this lemma, we can now prove Theorem \ref{Overtwisted-Group}, i.e. we can prove that \(OT_k(M)\) is the image of \(\pi_k||X_\bullet||\) induced by the natural forgetful map \(||X_\bullet|| \to \Co^{OT}(M)\) when \(1 \leq k \leq 2n\):
\begin{proof}
Recall the commuative diagram relating the different semi-simplicial spaces:
\[\begin{tikzcd}
||||X_\bullet|| \arrow[d] \arrow[r] & ||||X_\bullet^c|| \arrow[d] \arrow[r] & \text{Cont}^{OT}(M) \arrow[d] \\
||||Y_\bullet|| \arrow[r]           & ||||Y_\bullet^c|| \arrow[r]           & \text{AlmCont(M)}       
\end{tikzcd}\]
First, we will show that \(OT_k(M)\) is a subgroup of the image of \(\pi_k||X_\bullet||\). By abuse of notation we will use specific representatives of elements of homotopy groups when we mean their homotopy classes, so we are really working up to homotopy relative to the basepoint. Let \(\alpha: S^k \to \Co^{OT}(M)\) be an element of \(OT_k(M)\) and \(\Delta: S^k \to Emb(\bD^{2n},M)\) be a certificate of overtwistedness for \(\alpha\). Then \((\alpha,\Delta,1) \in \pi_k||X_\bullet^c||\) maps to \(\alpha\). However, we know that the maps \(||X_\bullet|| \to ||Y_\bullet||\) and \(||X_\bullet^c|| \to ||Y_\bullet^c||\) in the above diagram are weak equivalences. Furthermore, by the previous lemma \(||Y_\bullet^c|| \to \AC(M)\) is a weak equivalence, and since we know that \(||Y_\bullet|| \to \AC(M)\) is also weak equivalence by Theorem \ref{Section-Space-Weak-Equivalence} we have that \(||Y_\bullet|| \to ||Y_\bullet^c||\) is a weak equivalence. Combining these equivalences with the previous commutative diagram, we have that \(||X_\bullet|| \to ||X_\bullet^c||\) is a weak equivalence, and so there exists some \(\beta \in \pi_k||X_\bullet||\) such that \(\beta \mapsto (\alpha,\Delta,1)\) and hence maps to \(\alpha\). So indeed,  \(OT_k(M)\) is a subgroup of the image of \(\pi_k||X_\bullet||\). However, we know that the isomorphism \(OT_k(M) \to \pi_k\AC(M)\) and the isomorphism \(\pi_k||X_\bullet|| \to \pi_k \AC(M)\) are both induced by the natural inclusion \(\Co^{OT}(M) \to \AC(M)\), so we have a group isomorphism that remains an isomorphism when restricted to a subgroup. This is only possible if the subgroup is the whole group, so indeed \(OT_k(M)\) is this image.  
\end{proof}

\section{Infinite Cyclic Subgroups in the Homotopy Groups of the Contactomorphism Group}
\label{sec:five}

We can now use Theorem \ref{Improved-Contact-Geometry} to generalize the results from \cite{Fernandez-Gironella}.
Let \((M,\xi_{OT})\) be an overtwisted, closed, cooriented contact manifold of dimension \(2n+1\) and let \(\cC(M,\xi_{OT})\) be the contactomorphism group of \((M,\xi_{OT})\), i.e. all diffeomorphisms of \(M\) that preserve the contact structure \(\xi_{OT}\). Then from \cite[Lemma 1.1]{Giroux}  we have a fiber sequence
\[\cC_0(M,\xi_{OT}) \to \D_0(M) \to \Co(M),\]
where \(\D_0(M) \to \Co(M)\) is given by \(f \mapsto f^{\star} \xi_{OT}\),
which induces a long exact sequence of homotopy groups.
\begin{remark}
In the literature the fibration is given by pushforward not pullback, i.e. the map \(f \mapsto f_\star \xi_{OT}\). While this may be more natural geometrically, since we are using diffeomorphisms pullback is just pushforward by the inverse, so our map is still a fibration. Also, we will see later that it is convenient to factor this map through something more general, where pushforward is no longer well defined but pullback is.
\end{remark} 
We would like to find infinite cyclic subgroups inside \(\pi_k(\cC_0(M,\xi_{OT}),id)\), i.e. we want to find nonzero elements of the rational homotopy groups of \(\cC_0(M,\xi_{OT})\). To do this, we will prove a few lemmas. Let \(\text{Bun}(TM)\) denote the space of all pairs \((f,\delta f) \), where \(f:M \to M\) is a continuous map and \(\delta f: TM \to f^{\star}TM\) is some vector bundle map over \(M\) that is a fiberwise isomorphism, i.e. the space of bundle isomorphisms of \(TM\). 

\begin{lemma}\label{Diffeomorphism-Action-Factors}
The map 
\[\D_0(M) \to \Co(M) \to \text{AlmCont}(M)\]
factors through the space of bundle isomorphisms of \(TM\) as follows:
\[\begin{tikzcd}
\text{Diff}_0(M) \arrow[d] \arrow[rrr, "f \mapsto f^{\star}\xi_{OT}"]  &  &  & {\text{Cont}(M,\xi_{OT})} \arrow[d] \\
\text{Bun}(TM) \arrow[rrr, "{(f,\delta f) \mapsto (f,\delta f)^{\star}\xi_{OT}}"] &  & & \text{AlmCont}(M)                  
\end{tikzcd},
\]
where the left vertical map is the derivative \(f \mapsto (f,df)\), and \((f ,\delta f)^* \xi_{OT}\) is the almost contact structure obtained by realizing \(f^*\xi_{OT} \subset f^*TM\) as a sub-bundle of \(TM\) via the isomorphism \(\delta f: TM \to f^*TM\). 
\end{lemma}

\begin{proof}
It is clear the diagram commutes as long as the bottom map is well defined, so to verify this, we just need to ensure that if \((f,\delta f) \in \text{Bun}(TM)\) then \((f,\delta f)^\star \xi_{OT}\) is still an almost contact structure on \(M\). First, \(f^{\star} \xi_{OT}\), \(f^\star R\) are always a hyperplane bundle and line bundle on \(M\) respectively, for any continuous function \(f:M \to M\). Also, the Whitney sum decomposition and the almost complex structure are naturally preserved by pullback. The only possible obstruction is that \(f^\star \xi_{OT}\), \(f^{\star} R\) are not necessarily isomorphic to sub-bundles of \(TM\) for an arbitrary function. However, we are given a fiberwise isomorphism \(\delta(f): TM \to f^\star TM\) which allows us to realize these pullbacks as sub-bundles of the tangent bundle, as required.
\end{proof}

Next, we need some results about how the diffeomorphism group of a disk glued into \(M\) maps to \(\Co(M)\). First, we can use a recent result in \cite{Ebert} Theorem \(1.4\) which says that the inclusion map 
\[\D_\partial(\bD^{2n+1}) \to \D_0(M)\]
is injective on rational homotopy in degrees \(k\) in the rational concordance range \(k \leq \phi^\bQ(\bD^{2n}) - 1\), \(k \neq 0\). So, if we can find something nontrivial in \(\pi_k\D_\partial(D^{2n+1}) \otimes \bQ\) that maps to zero in \(\pi_k\Co(M) \otimes \bQ\) in this range, then by the injectivity result we will have a nontrivial element of \(\pi_k \D_0(M) \otimes \bQ\) that maps to zero in \(\pi_k\Co(M) \otimes \bQ\), which will give us a nontrivial element of \(\pi_k\cC_0(M,\xi_{OT}) \otimes \bQ\) by exactness. 

\begin{lemma}\label{Disk-Diffeomorphisms-Preserve-Contact-Structures}
Let \(\alpha \in \pi_k \D_\partial(\bD^{2n+1}) \) such that \(\alpha \mapsto 0\) under the map 
\[\pi_k \D_\partial(\bD^{2n+1}) \to \pi_k \D_0(M) \to \pi_k\Co(M,\xi_{OT}) \to \pi_k \AC(M).\]
Then we also have that \(\alpha \mapsto 0\) under the map 
\[\pi_k \D_\partial(\bD^{2n+1}) \to \pi_k \D_0(M) \to \pi_k\Co(M,\xi_{OT}) .\]
\end{lemma}

\begin{proof}
Let \( \D_\partial(\bD^{2n+1}) \to \Co(M)\) be given by composing through the diffeomorphisms of \(M\). Now, let \(\Delta\) be an overtwisted disk for \(\xi_{OT}\), and let \(D \simeq \bD^{2n+1}\) be an embedded disk in \(M\) disjoint from a neighborhood of \(\Delta\). Then since we get a diffeomorphism of \(M\) by extending a diffeomorphism of the disk by the identity, we have that \(\Delta\) is overtwisted for \(f^\star \xi_{OT} \) for all \(f \in \D_\partial(D)\). So, if we let 
\[||X_\bullet|| = ||\coprod_{p+1 \text{ disjoint overtwisted disks } (\Delta_0, \dots , \Delta_p)} \Co^{OT}(M \text{ rel } \Delta_0, \dots, \Delta_p)||\]
Then we have that the map \( \D_\partial(\bD^{2n+1}) \to \Co(M)\) factors through \(||X_\bullet||\), where \(f \mapsto (f^\star \xi_{OT}, \Delta, 1)\), and of course the map \(||X_\bullet|| \to \Co(M)\) just comes from forgetting the overtwisted disks. So, we have the following commutative diagram:
\[\begin{tikzcd}
\text{Diff}_\partial(\bD^{2n+1})  \arrow[rr] \arrow[d] &  & \text{Cont}(M) \arrow[d] \\
||||X_\bullet|| \arrow[rr] \arrow[rru]                 &  & \text{AlmCont}(M)       
\end{tikzcd}\]
But, as we saw in a previous section the map \(||X_\bullet|| \to \AC(M)\) is a weak equivalence, so indeed anything mapping to zero in the homotopy of \(\AC(M)\) must map to zero in the homotopy of \(||X_\bullet||\) and thus in \(\Co(M)\) as well. 
\end{proof}

Now, we have enough to prove the following lemma:
\begin{lemma}\label{Disk-Diffeomorphisms-Act-Rationally-Null-Homotopically}
The map \(\D_\partial(\bD^{2n+1}) \to \D_0(M) \to \Co(M)\) is trivial on rational homotopy groups.
\end{lemma}
\begin{proof}
From Lemma \ref{Diffeomorphism-Action-Factors}, we can factor the map through bundle isomorphisms, as follows:
\[\begin{tikzcd}
\text{Diff}_\partial(\bD^{2n+1}) \arrow[rr] \arrow[d] &  & \text{Diff}_0(M) \arrow[rr] \arrow[d] &  & \text{Cont}(M) \arrow[d] \\
\text{Bun}_\partial(\bD^{2n+1}) \arrow[rr]            &  & \text{Bun}(TM) \arrow[rr]             &  & \text{AlmCont}(M)       
\end{tikzcd}
\]
where the left vertical map is the derivative 
\(d: \D_\partial(\bD^{2n+1}) \to \Omega^{2n+1} SO(2n+1)\). But, this map is zero on rational homotopy groups, see for example \cite{Crowley} page nine. So, the composition through the bottom of the diagram to \(\AC(M)\) is zero on rational homotopy. So, since the diagram is commutative, we can go through the top, and use Lemma \ref{Disk-Diffeomorphisms-Preserve-Contact-Structures} to conclude \(\D_\partial(\bD^{2n+1}) \to \D_0(M) \to \Co(M)\) is trivial on rational homotopy groups, as required. 
\end{proof}
\begin{corollary}\label{Infinite-Cyclic-Subgroups-of-Contactomorphisms}
If \((M,\xi_{OT})\) is a closed, cooriented, overtwisted contact manifold of dimension \(2n+1\), then \(\pi_k\cC_0(M,\xi_{OT})\) contains an infinite cyclic subgroup whenever \(\pi_k\D_\partial(\bD^{2n+1}) \otimes \bQ \neq 0\), for \(k \leq  \phi^\bQ(\bD^{2n})-1\), \(k \neq 0\). 
\end{corollary}
\begin{proof}
By applying the previous lemma and the injectivity result of \cite{Ebert}, every nonzero element of \(\pi_k\D_\partial(\bD^{2n+1}) \otimes \bQ\) maps to \(0\) in \(\pi_k \Co^{OT}(M) \otimes \bQ\) and so by exactness, must come from something  nonzero in \(\pi_k \cC_0(M,\xi_{OT}) \otimes \bQ\).
\end{proof}

\section{Further Applications}
\label{sec:six}
Another immediate application to Theorem \ref{Relative-H-Principle} appears in Engel geometry, where it was recently shown that there is a notion of overtwistedness parallel to contact overtwistedness, and one still gets an \(h\)-principle with a fixed overtwisted disk, see \cite{del_Pino} (Theorem 1.1 and Corollary 1.2). All of the relevent properties of contact overtwistedness also apply to Engel overtwistedness; The collection of all overtwisted engel disks is still sufficiently separated, and Theorem 1.1 of \cite{del_Pino} gives a strong enough relative \(h\)-principle that we can get an \(h\)-principle for any number of fixed overtwisted disks. We can apply Theorem \ref{Relative-H-Principle} to conclude that for a \(4\)-manifold \(M\), \(\cE(M) \to \cE^f(M)\) admits a section up to homotopy, where \(\cE(M), \cE^f(M)\) are the spaces of Engel and formal Engel structures on \(M\), respectively. This shows that \(\pi_k\cE^f(M)\) is a subgroup of \(\pi_k\cE(M)\) for all \(k\) via the map from the semi-simplicial realization. Furthermore, using the foliated version of this \(h\)-principle from \cite{del_Pino} Theorem 6.25, this subgroup agrees with the subgroup found using a certificate of overtwistedness in degrees \(k \leq 3\) in \cite{del_Pino}, using the same arguments used in Theorem \ref{Overtwisted-Group}. It was already known that the map \(\cE(M) \to \cE^f(M)\) is surjective on homotopy groups by \cite{Perez}, and one of the main results of \cite{Casals} shows \(\pi_k\cE^f(M)\) is a subgroup of \(\pi_k\cE(M)\) for all \(k\). The natural question this raises is whether the subgroup of \(\pi_k\cE(M)\) found in \cite{Casals} using loose Engel structures is the same as the subgroup one gets using semi-simplicial realization via overtwisted disks. Understanding this may help to understand how loose and overtwisted Engel structures interact, which is currently poorly understood.

\bibliographystyle{amsalpha}
\bibliography{./refs}

\end{document}